\documentclass{amsart}
\title[Finiteness for some $\mathbb{Q}$-PD groups]{Finiteness properties for some rational Poincar\'e duality groups}
\author{Jim Fowler}
\address{The Ohio State University \\ 100 Math Tower \\ 231 West 18th Avenue \\ Columbus, OH 43210--1174}
\email{fowler@math.ohio-state.edu}


\usepackage{hyperref}
\usepackage{amsmath}
\usepackage{amssymb}
\usepackage{ulem}
\usepackage{mathabx}
\usepackage{MnSymbol}

\newcommand{\Z}{\mathbb{Z}}
\newcommand{\Q}{\mathbb{Q}}

\newcommand{\R}{\mathbb{R}}

\DeclareMathOperator{\CAT}{CAT}
\DeclareMathOperator{\SO}{SO}
\DeclareMathOperator{\trace}{tr}

\DeclareMathOperator{\proj}{proj}

\DeclareMathOperator{\PD}{PD}

\DeclareMathOperator{\FP}{FP}

\DeclareMathOperator{\FL}{FL}
\DeclareMathOperator{\FH}{FH}
\DeclareMathOperator{\V}{V}

\usepackage[colorinlistoftodos, textwidth=4cm]{todonotes}

\usepackage{amsthm}
\newtheorem{theorem}{Theorem}[section]
\newtheorem*{maintheorem}{Main Theorem}
\newtheorem{corollary}[theorem]{Corollary}

\newtheorem*{selberglemma}{Selberg's Lemma}
\newtheorem{proposition}[theorem]{Proposition}

\theoremstyle{remark}

\newtheorem{example}[theorem]{Example}
\newtheorem{question}[theorem]{Question}

\theoremstyle{definition}
\newtheorem{definition}[theorem]{Definition}

\DeclareMathOperator{\id}{id}

\DeclareMathOperator{\op}{op}
\DeclareMathOperator{\Obj}{Obj}
\DeclareMathOperator{\Or}{Or}
 
\newcommand{\category}[1]{\mbox{\textbf{#1}}}

\DeclareMathOperator{\Hom}{Hom}

\newcommand{\defnword}[1]{\textit{#1}\index{#1}}

\DeclareMathOperator{\Wall}{Wall}
\newcommand{\ReducedWall}{\widetilde{\Wall}}
\DeclareMathOperator{\Span}{span}

\newcommand{\point}{\mbox{point}}

\newcommand{\Zmod}[1]{\Z/#1\Z}
\newcommand{\GmodKmodGamma}{\Gamma \backslash G/K}

\begin{document}

\begin{abstract}
  A combination of Bestvina--Brady Morse theory and an
  \textit{acyclic} reflection group trick produces a torsion-free
  finitely presented $\mathbb{Q}$-Poincar\'e duality group which is
  not the fundamental group of an aspherical closed ANR
  $\mathbb{Q}$-homology manifold.

  The acyclic construction suggests asking which
  $\mathbb{Q}$-Poincar\'e duality groups act freely on
  $\mathbb{Q}$-acyclic spaces (i.e., which groups are
  $\mbox{FH}(\mathbb{Q})$).  For example, the orbifold fundamental
  group $\Gamma$ of a good orbifold satisfies $\mathbb{Q}$-Poincar\'e
  duality, and we show $\Gamma$ is $\mbox{FH}(\mathbb{Q})$ if the
  Euler characteristics of certain fixed sets vanish.
\end{abstract}

\maketitle

\section{Introduction}

Existence and uniqueness questions (i.e., ``the Borel conjecture'')
for closed aspherical $\Z$-homology manifolds can be formulated for
$R$-homology manifolds.  Mike Davis does this in \cite{MR1747535}: he
asks if some algebra (i.e., having $R$-Poincar\'e duality) is
necessarily a consequence of some geometry (i.e., being an
$R$-homology manifold).
\begin{question}[M.~Davis]
  \label{question:davis}
  Is every torsion-free finitely presented group satisfying
  $R$-Poincar\'e duality the fundamental group of an aspherical closed
  $R$-homology $n$-manifold?
\end{question}
\noindent
Theorem~\ref{theorem:torsion-free-counterexample} answers this
question in the negative.  However, the construction of that
counterexample, as well as the spirit of the original question,
suggests weakening the conclusion.
\begin{question}[Acyclic variant of a question of M.~Davis]
  \label{question:davis-modified}
  Suppose $\Gamma$ is a finitely presented group satisfying
  $R$-Poincar\'e duality.  Is there a closed $R$-homology
  manifold $M$, with
\begin{itemize}
\item $\pi_1 M = \Gamma$, and
\item $H_\star(\tilde{M};R) = H_\star(\point;R)$, that is, $R$-acyclic universal cover?
\end{itemize}
\end{question}
\noindent
Instead of asking for an aspherical homology manifold (as in the
original question), this modified question only asks that the homology
manifold have $R$-acyclic universal cover.  Nevertheless, a group
acting geometrically and cocompactly on an $R$-acyclic $R$-homology
manifold still possesses $R$-Poincar\'e duality, so an affirmative
answer to Question~\ref{question:davis-modified} provides a
``geometric source'' for the $R$-Poincar\'e duality of a group.

For example, if a finite group $G$ acts on a manifold $B\pi$, then
there is an extension
$$
1 \to \pi \to \Gamma \to G \to 1, \mbox{ with $\Gamma = \pi_1\left( (EG
  \times B\pi)/G \right)$,}
$$
and $\Gamma$ satisfies $\Q$-Poincar\'e duality.  But for $\Gamma$ to
be the fundamental group of a $\Q$-homology manifold with acyclic
universal cover requires, in particular, that $\Gamma$ be the
fundamental group of a \textit{finite complex} with $\Q$-acyclic
universal cover.  Thus we are led to ask 
\begin{question}
  For which groups $\Gamma$ does there exist a finite complex $X$ with
\begin{itemize}
\item $\bar{H}_\star(\tilde{X};R) = 0$ and
\item $\pi_1 X = \Gamma$?
\end{itemize}
\end{question}
\noindent In other words, which groups act ``nicely'' (e.g., properly
discontinuously, cellularly, cocompactly) on acyclic complexes?  This
is property $\FH(R)$ introduced by M.~Bestvina and N.~Brady.

\nocite{MR620445}\nocite{MR546920}

Answering this question in the context of orbifolds amounts to a finiteness obstruction.  W.~L\"uck designed an
equivariant finiteness theory \cite{MR1027600} (and there are other
descriptions of equivariant finiteness obstructions in the literature
\cite{MR620445,MR546920}).  In
section~\ref{section:modules-over-categories} we will describe an
equivariant finiteness theory in a setup similar to that in \cite{MR1659969}, with the following result:
\begin{maintheorem}
  Suppose a finite group $G$ acts on a finite complex $B\pi$; let $\Gamma$
  be $\pi_1\left((EG \times B\pi)/G\right)$, that is, the orbifold fundamental
  group of $(B\pi)/G$.

  If, for all nontrivial subgroups $H \subset G$, and every connected
  component $C$ of $(B\pi)^H$,
$$
\chi\left( C \right) = 0,
$$
then $\Gamma \in \FH(\Q)$, i.e., there exists a finite CW complex $X$ with $\pi_1 X = \Gamma$ and $\tilde{X}$ rationally acyclic.
\end{maintheorem}
In section~\ref{subsection:lefschetz}, we will also see that the vanishing
of certain Euler characteristics is necessary, namely $\chi( (B\pi)^H
)$ for cyclic subgroups $H$.  There are some examples in 
section~\ref{section:examples}.

\subsection*{Acknowledgments}

This paper grew out of my Ph.D.~thesis; I am very thankful for all
the help that my thesis advisor, Shmuel Weinberger, has given me.  I
also thank the referee for many detailed and helpful comments which
improved the paper.

\section{Reflection group trick}
\label{section:reflection-group-trick}

In \cite{MR1600586}, M.~Davis combined Bestvina--Brady Morse theory
with the reflection group trick to produce Poincar\'e duality groups
that are not finitely presented---and therefore, not fundamental
groups of aspherical manifolds.  We apply this technique to rational
Poincar\'e duality groups and rational homology manifolds, answering
Question~\ref{question:davis} in the negative.
\begin{theorem}
\label{theorem:torsion-free-counterexample}
There exists a torsion-free, finitely presented $\PD(\Q)$-group
$\Gamma$ which is not the fundamental group of an aspherical closed
ANR $\Q$-homology manifold.
\end{theorem}
\noindent
The construction of such a group $\Gamma$ proceeds as follows:
\begin{itemize}
\item Let $X$ be a simply connected finite complex which is
  $\Q$-acyclic but not $\Z$-acyclic (for concreteness, the CW-complex
  $S^2 \cup e^3$ where the attaching map is a degree two map $S^2 \to
  S^2$).
\item Apply Bestvina--Brady Morse theory \cite{MR1465330} to $X$; this
  produces a group $G \not\in\FP(\Z)$ with $G \in \FH(\Q)$, so $G$
  acts freely and cocompactly on a $\Q$-acyclic space; let $K$ be the
  quotient of such a free action.
\item Apply a variant of M.~Davis' reflection group trick
  \cite{MR690848} to a thickened version of $K$; after taking a cover,
  this produces a torsion-free group $\Gamma$ satisfying $\PD(\Q)$.
\item Verify that the space $B\Gamma$ is not homotopy equivalent to a
  finite complex.
\end{itemize}
A closed ANR $\Q$-homology manifold is homotopy equivalent to a finite
complex \cite{MR451247}.  Since $B\Gamma$ is not homotopy
equivalent to a finite complex, $\Gamma$ cannot be the fundamental
group of an aspherical closed $\Q$-homology manifold.

\subsection{Proof of Theorem~\ref{theorem:torsion-free-counterexample}}

We
exhibit a torsion-free, finitely presented $\PD(\Q)$-group which is
not the fundamental group of an aspherical finite complex, let alone a
closed ANR $\Q$-homology manifold.

\subsubsection{PL~Morse theory}

We begin by producing a group which satisfies a rational finiteness
property, but not an integral finiteness property.  Choose a simply
connected finite complex $X$ which is $\Q$-acyclic but not
$\Z$-acyclic; for the sake of concreteness,
$$
X = S^2 \cup_f e^3 \mbox{ with $f : \partial e^3 = S^2 \to S^2$ degree two.}
$$
Let $L$ be a flag triangulation of $X$.  Then consider the union of tori
$$
X = \bigcup_{\sigma \in L} \prod_{v \in \sigma} S^1.
$$
In the cube complex $X$, it is easy to check that the link of each
vertex is $L$, and, because $L$ is flag, this means that $X$ is
$\CAT(0)$ (see \cite{MR1744486}).

The following theorem summarizes a result of Bestvina--Brady PL~Morse
theory \cite{MR1465330}; this is a version of Morse theory designed to
analyze spaces such as the above cubical complex.
\begin{theorem}
  Let $L$ be a finite flag complex.  Let $A = A_L$ the associated
  right angled Artin group, and $\Gamma$ the kernel of a natural map $A_L \to \Z$.
\begin{itemize}
\item If $L$ is $R$-acyclic, then $\Gamma \in \FH(R)$.
\item If $L$ is simply connected, then $\Gamma$ is finitely presented.
\item If $L$ is not $R$-acyclic, then $\Gamma \not\in \FP(R)$.
\end{itemize}
\end{theorem}
Since $L$ is simply connected and $\Q$-acyclic, but not $\Z$-acyclic,
the corresponding $\Gamma$ is finitely presented and $\FH(\Q)$, but
not $\FP(\Z)$.

The fact that $\Gamma \in \FH(\Q)$ means there is a finite complex $K$
with $\pi_1 K = \Gamma$ and $\bar{H}_\star(\tilde{K};\Q) = 0$; it is
to this complex $K$ that we apply the reflection group trick.

\subsubsection{An acyclic reflection group trick}
\label{subsubsection:reflection-group-trick}
\label{subsubsection:acyclic-reflection-group-trick}

Mike Davis introduced his reflection group trick in \cite{MR690848};
his excellent book \cite{MR2360474} is a great introduction to the technique.

\nocite{MR350744}
\nocite{MR248802}

Since $K$ is a finite complex, there is an embedding $K
\hookrightarrow \R^N$ for some $N$; a \defnword{regular neighborhood}
of $K \subset \R^N$ is manifold $N$ with boundary $\partial N$.
Observe that $N$ deform retracts to $K$.  An introduction to the
theory of regular neighborhoods can be found in \cite{MR350744,MR248802}.

The reflection group trick uses a Coxeter group to glue together
copies of $N$, transforming the manifold with boundary $N$ to a closed
manifold $W$.  We now describe how the copies of $N$ are glued
together.  Choose a flag triangulation $L$ of $\partial N$; let $G$ be
the right-angled Coxeter group associated to this flag triangulation.
For each vertex $v \in L$, let $D_v$ be the star of $v$ in the
barycentric subdivision $L'$ of $L$.  Copies of $N$ will be glued
along the ``mirrors'' $D_v$; specifically, define
$$
\tilde{W} = (N \times G) / \sim
$$
where $(x,g) \sim (x,gh)$ whenever $x \in D_h$.  Choose a finite
index torsion-free subgroup $G'$ of $G$, and let $W = \tilde{W} / G'$.
An application of Mayer-Vietoris proves
\begin{proposition}
$W$ is a closed manifold, with $\Q$-acyclic universal cover.
\end{proposition}
Additionally, $\pi_1 W = G' \rtimes \Gamma$ is torsion-free, since
$G'$ and $\Gamma$ are both torsion-free (the former by assumption, the
latter because it is a subgroup of an Artin group).

\begin{proposition}
  The space $B\pi_1 W$ does not have the homotopy type of a finite
  complex.
\end{proposition}
\begin{proof}
The group $\pi_1 W = G' \rtimes \Gamma$ retracts onto $\Gamma$, and
so, $B\pi_1 W$ retracts onto $B\Gamma$.  If $B\pi_1 W$ had
the homotopy type of a finite complex, then $B\Gamma$ would be a
finitely dominated complex, and so $\Gamma \in \FP(\Z)$.  But $\Gamma$
was constructed above (using Bestvina--Brady Morse theory) so that
$\Gamma\not\in\FP(\Z)$.
\end{proof}
The existence of such a group $W$ proves
Theorem~\ref{theorem:torsion-free-counterexample}, answering
Question~\ref{question:davis} in the negative.

\section{Modules over categories}
\label{section:modules-over-categories}

\newcommand{\Spaces}{\category{Spaces}}
\newcommand{\PointedSpaces}{\category{PointedSpaces}}
\newcommand{\FindomSpaces}{\category{FindomSpaces}}
\newcommand{\Groupoids}{\category{Groupoids}}
\newcommand{\Cat}{\category{Cat}}
\newcommand{\Rings}{\category{Rings}}
\newcommand{\AbGroups}{\category{AbGroups}}
\newcommand{\Complexes}[1]{\category{Cplx}\left(#1\right)}
\newcommand{\Groups}{\category{Groups}}
\newcommand{\Modules}[1]{\mbox{$#1$-$\category{Mod}$}}
\newcommand{\Sets}{\category{Sets}}
\newcommand{\CC}[1]{\mbox{$\mathcal{C}$-$#1$}}

We work in the framework used by Davis--L\"uck's of ``spaces over a
category'' \cite{MR1659969}, using the equivariant algebraic
$K$-theory developed by L\"uck \cite{MR1027600}.  After setting up
this framework, we will describe an instant finiteness obstruction in
section~\ref{subsection:instant}; this obstruction lies in $K_0(
\Groupoids \downarrow \Modules{R} )$, in contrast to the ``usual''
obstruction which lies in $K_0(R\Gamma)$.  The advantage to
considering the refined obstruction in $K_0( \Groupoids \downarrow
\Modules{R} )$ is that it can be computed in terms of Euler
characteristics of components of the fixed sets, as we'll see in
section~\ref{subsection:sufficient-condition}.

\subsection{Categories over categories}

\begin{definition}
\label{definition:category-over-c}
Let $\mathcal{C}$ be a small category, and $\Cat$ any category; we
define a category of diagrams $\CC{\Cat}$ as follows:
\begin{itemize}
\item $\Obj \CC{\Cat}$ consists of functors $F : \mathcal{C} \to \category{Cat}$, and
\item given two such functors $F$ and $G$, the morphisms $\Hom_{\CC{\Cat}} \left( F, G \right)$ are the natural transformations from $F$ to $G$.
\end{itemize}
\end{definition}
This is the \defnword{functor category} and is usually denoted
$\category{Cat}^{\mathcal{C}}$; we use the alternate notation
$\CC{\Cat}$, evoking the equivariant notation as in ``$G$-spaces.''

\begin{example}
The category $\Spaces$ is the category of compactly generated topological
spaces; an object in $\CC{\Spaces}$ is called a (covariant)
$\mathcal{C}$-space; a contravariant $\mathcal{C}$-space is a
covariant $\mathcal{C}^{\op}$-space.  We likewise have
$\CC{\AbGroups}$ and $\CC{\Modules{R}}$, which form abelian categories
and for which W.~L\"uck has developed homological algebra
\cite{MR1027600}.
\end{example}

Any functor $F : \category{A} \to \category{B}$ induces
$$
\CC{F} : \CC{\category{A}} \to \CC{\category{B}}
$$
by sending $A : \mathcal{C} \to \category{A}$ to $\CC{F}(A) = F \circ
A$.

For instance, the fundamental groupoid functor $\Pi : \Spaces \to
\Groupoids$ induces
$$
\CC{\Pi} : \CC{\Spaces} \to \CC{\Groupoids}.
$$
At times, however, we would like to talk about the category of
$\mathcal{C}$-spaces, for a varying small category $\mathcal{C}$.
This desire is behind the following definition.
\begin{definition} 
Let $\category{A}$ and $\category{B}$ be categories, with
$\category{A}$ a subcategory of $\Cat$, the category of small
categories.  Then the category
\[
\category{A} \downarrow \category{B}
\]
has as objects the functors $F : A \to \category{B}$, for $A$ an
object of $\category{A}$; in other words, an object of the category
$\category{A} \downarrow \category{B}$ consists of a choice of an
object $A \in \category{A}$, and a functor from $A$ to $\category{B}$.

A morphisms in $\Hom_{\category{A} \downarrow \category{B}}(F : A \to \category{B}, F' : A' \to \category{B})$
consists of a functor $H : A \to A'$ with a natural transformation from
$F$ to $F' \circ H$.

Although it will not be important in the sequel, note that
$\category{A}$ is a $2$-category (meaning a category enriched over
$\Cat$), and $\category{A} \downarrow \category{B}$ is likewise a
$2$-category.
\end{definition}

\subsection{Balanced products}

A construction well-known to category theorists---that of a
\textit{coend}---gives a natural transformation from a bifunctor
$\mathcal{C}^{\op} \times \mathcal{C} \to \Cat$ to a constant functor
\cite{MR354798}.  We apply this in the case of $\Cat$, a monoidal
category, to combine a contravariant and covariant
$\mathcal{C}$-object over $\Cat$ into an object of $\Cat$.

\begin{definition}
  Let $\Cat$ be a monoidal category with product $\times$; let $A$ and
  $B$ be contravariant and covariant $\mathcal{C}$-objects,
  respectively. Then the \defnword{balanced product} of $A$ and $B$,
  written $A \times_{\mathcal{C}} B$, is
$$
\bigsqcup_{c \in \Obj \mathcal{C}} A(c) \times B(c) / \sim
$$
where $(xf,y) \sim (x,fy)$ for $x \in A(d)$, $y \in B(c)$, and $f \in
\Hom(c,d)$.
\end{definition}
We will be using balanced products in the context of spaces (under
cartesian product of spaces) and modules (under tensor product of
modules).  At first, balanced products may seem too abstract, but balanced products
(and coends more generally) are abstractions of a better-known
construction: geometric realization.
\newcommand{\SimplexCategory}{\mathbf{\Delta}}
\begin{example}
\label{example:geometric-realization}
  Define $\SimplexCategory$, the simplicial category (see
  \cite{MR222892}), where
\begin{itemize}
\item $\Obj \SimplexCategory$ consists of totally ordered finite sets, and
\item $\Hom_{\SimplexCategory}(A,B)$ consists of order-preserving
  functions from $A$ to $B$.
\end{itemize}
Further define $\Delta$ to be the $\SimplexCategory$-space, sending a totally ordered finite $A$ to
$$
\Delta(A) = \mbox{$(|A| - 1)$-simplex},
$$
and an order-preserving function to the inclusion of simplices.

A \defnword{simplicial space} is a functor $X : \SimplexCategory^{\op}
\to \category{Spaces}$, i.e., an object of
$\SimplexCategory^{\op}$-$\Spaces$.  The balanced product of a
simplicial space $X$ with $\Delta$ (written $X
\times_{\SimplexCategory} \Delta$), is the \defnword{geometric
  realization} of the simplicial space $X$.
\end{example}

\subsection{Orbit category}

\begin{definition}
  Define the \defnword{orbit category} of a group $G$, written $\Or(G)$, as follows:
\begin{itemize}
\item $\Obj \Or(G) = \{ G/H : \mbox{$H$ a subgroup of $G$} \}$,
\item $\Hom_{\Or(G)}(G/H, G/K)$ is the set of $G$-maps between the $G$-sets $G/H$ and $G/K$.
\end{itemize}
\end{definition}
Naturally associated to a $G$-space, there are both contravariant
and covariant $\Or(G)$-spaces.
\begin{example}
Let $X$ be a (left) $G$-space; there is a contravariant $\Or(G)$-space
$$
G/H \mapsto X^H,
$$
with $G/H \to G/H'$ sent to $X^{H'} \subset X^{H}$.
Associated to $X$, there is also a covariant $\Or(G)$-space
$$
G/H \mapsto X/H,
$$
with $G/H \to G/H'$ sent to $X/H \to X/H'$.
\end{example}
In fact, the reverse is possible: given a contravariant
$\Or(G)$-space, we can recover a $G$-space.
\begin{proposition}
\label{proposition:recognize-spaces-with-actions}
A contravariant $\Or(G)$-space is (naturally) a left $G$-space.
\end{proposition}
\begin{proof}
  The construction is formally similar to geometric realization (see
  Example~\ref{example:geometric-realization}).

  Suppose $X$ is a contravariant $\Or(G)$-space.  Let $\nabla$ be the
  covariant $\Or(G)$-space given by sending $G/H$ to itself, that is,
  to the finite set with the discrete topology.  Then
$$
X \times_{\Or(G)} \nabla
$$
is a (left) $G$-space.  Specifically, $g \in G$ acts on $X
\times_{\Or(G)} \nabla$ by the map $\id \times_{\Or(G)} L_g$ where
$L_g : G/H \to G/H$ is left multiplication by $g$.
\end{proof}

\subsection{$K$-theory}

An object in $\Groupoids \downarrow \Modules{R}$ is an
``$R[G]$-module'' for some groupoid $G$; we define certain (full)
subcategories of $\Groupoids \downarrow \Modules{R}$, corresponding to
finitely generated free and finitely generated projective
$R[G]$-modules.  We will speak of both contravariant and covariant
$R[G]$-modules.

\begin{definition}
  A \defnword{complex} in $\Groupoids \downarrow \Modules{R}$ is a
  collection of such modules $M_i$ (with $i\in \Z)$ and maps $d_i :
  M_i \to M_{i-1}$.  We say the complex is \defnword{bounded} if all
  but finitely many of the modules are zero.

  Write $\Complexes{ \Groupoids \downarrow \Modules{R} }$ for the
  category of complexes of finitely generated projective $R$-modules
  over a groupoid; maps between complexes are \defnword{chain maps}.
\end{definition}

As is usually the case, ``free'' is adjoint to ``forgetful'' (i.e.,
the forgetful functor from $\Groupoids \downarrow \Modules{R}$ to
$\Groupoids \downarrow \Sets$).
\begin{definition}[see page 167, \cite{MR1027600}]
  A module $M$ in $\Groupoids \downarrow \Modules{R}$ is a
  \defnword{free module} with basis $B \subset M$, an object in
  $\Groupoids \downarrow \Sets$, if, for any object $N$ in $\Groupoids
  \downarrow \Modules{R}$ and map $f : B \to N$, there is a unique
  morphism $F : M \to N$ extending $f$.
\end{definition}

In addition to free modules with basis $B$, we can speak about modules
generated by a particular subset.
\begin{definition}[see page 168, \cite{MR1027600}]
  Suppose $M$ is an object in the category $\Groupoids \downarrow \Modules{R}$, and
  $S$ is a subset (i.e., an object in $\Groupoids \downarrow \Sets$).
  Then the \defnword{span} of $S$ is the smallest module containing $S$, namely,
$$
\Span S = \bigcap \{ N : S \subset N \mbox{ and $N$ is a submodule of $M$  } \}.
$$
If $S$ is a finite set (i.e., finite over the indexing category,
meaning $S(g)$ is a finite set for each object $g$ in the groupoid),
we say that $\Span S$ is \defnword{finitely generated}.
\end{definition}

\begin{definition}[see page 169, \cite{MR1027600}]
  A module $P$ in $\Groupoids \downarrow \Modules{R}$ is
  \defnword{projective} if either of the following equivalent
  conditions holds:
\begin{itemize}
\item Each exact sequence $0 \to M \to N \to P \to 0$ splits.
\item $P$ is a direct summand of a free module.
\end{itemize}
\end{definition}

Having studied these modules, we can define an appropriate $K$-theory
for the category $\Groupoids \downarrow \Modules{R}$, via Waldhausen
categories \cite{MR802796} as in \cite{MR1027600}.

This $K$-theory is the correct receiver for the Euler characteristic.
\begin{definition}
  The \defnword{Euler characteristic} $\chi$ of a bounded complex $(M_i,d_i)$ in
  $\Complexes{\Groupoids \downarrow \Modules{R}}$ is 
$$
\chi \left( \cdots \to M_0 \to \cdots \right) = \sum_{i \in \Z} (-1)^i [M_i] \in K_0( \Groupoids \downarrow \Modules{R} ).
$$
\end{definition}

\subsection{Chain complex of the universal cover}

In Wall's finiteness obstruction for a space $X$, the most important
object is $\tilde{C}(X)$, the $R[\pi_1 X]$-chain complex of the
universal cover of $X$.  This is traditionally denoted by
$C_\star(\tilde{X} ; R)$, but we will write $\tilde{C}(X)$ to
emphasize the functorial nature of the construction.

However, the usual construction is insufficiently functorial:
$\tilde{C}$ transforms a space $X$ into a chain complex over a ring
that depends on the group $\pi_1 X$; consequently, it is not clear
what the target category of $\tilde{C}$ ought to be.  Worse, only
basepoint preserving maps $X \to X$ induce endomorphisms of
$\tilde{C}(X)$.

The definition of $\category{A} \downarrow \category{B}$ is exactly
what we need to define the target of the functor $\tilde{C}$, and by
using the fundamental groupoid instead of the fundamental group, we
avoid the basepoint issue: \textit{any} self-map of $X$ will induce a
self-map of $\tilde{C}(X)$.

Before we can define $\tilde{C}$, we define the \defnword{universal
  cover functor}.  The functor $\tilde{-} : \Spaces \to \Groupoids
\downarrow \Spaces$ sends a space $X$ to the functor $\tilde{X} : \Pi
X \to \Spaces$.  This latter functor sends a object in $\Pi X$, which
is just a point $x \in X$, to the universal cover of $X$ using $x$ as
the base point.

The functor $C : \Spaces \to \Complexes{\Modules{R}}$ sends a space to
its singular $R$-chain complex.  Note that this induces a functor
$$
\Groupoids \downarrow C : \Groupoids \downarrow \Spaces \to \Groupoids \downarrow \Complexes{\Modules{R}}
$$
Not too surprisingly, we compose $\tilde{-}$ and $\Groupoids \downarrow C$.
\begin{definition}[See page 259, \cite{MR1027600}]
  The functor
$$
\tilde{C} : \Spaces \to \Complexes{\Groupoids \downarrow \Modules{R}}
$$
sends a space $X$ to $C(\tilde{X})$.

\end{definition}
Note that there is a natural map
$$
\Groupoids \downarrow \Complexes{\Modules{R}} \to \Complexes{\Groupoids \downarrow \Modules{R}}.
$$
As a result of the functoriality of $\tilde{C}$, we can apply
$\tilde{C}$ over a small category $\mathcal{C}$

to get
\begin{align*}
\CC{\tilde{C}} : \CC{\Spaces} &\to \CC{\Complexes{\Groupoids \downarrow \Modules{R}}} \\
                              &\to \Complexes{\mathcal{C} \downarrow \left( \Groupoids \downarrow \Modules{R} \right)}.
\end{align*}

\subsection{Instant finiteness obstruction}
\label{subsection:instant}

Our goal is to define maps
\begin{align*}
\Wall &: \FindomSpaces \to K_0( \Groupoids \downarrow \Modules{R} ), \\
\ReducedWall &: \FindomSpaces \to \tilde{K}_0( \Groupoids \downarrow \Modules{R} ),
\end{align*}
so that $\ReducedWall \neq 0$ obstructs an $R$-finitely dominated
space from being $R$-homotopy equivalent to a finite complex.  There
are a few terms that need to be defined.

Here, $\FindomSpaces$ is built from a full subcategory of $\Spaces$,
consisting of those spaces which are $R$-\defnword{finitely
  dominated}, but the choice of domination is part of the data.
\begin{definition}
A space $Y$ is $R$-\defnword{dominated} by $X$ if there are maps
$$
Y \stackrel{i}{\longrightarrow} X \stackrel{r}{\longrightarrow} Y
$$
with $r \circ i : Y \to Y$ an $R$-homotopy equivalence.

Further, a space is $R$-\defnword{finitely dominated} by $X$ if $X$ is a finite
complex.
\end{definition}
Whenever we speak of an $R$-homotopy equivalence, we really mean an
$R[\pi_1]$ equivalence---i.e., the induced map $\tilde{C}(Y) \to
\tilde{C}(Y)$ is chain homotopic to the identity.

Ranicki has defined an \textit{instant finiteness obstruction}.  His
algebraic framework remains applicable for complexes of modules over a
category.  In particular, Proposition~3.1 of \cite{MR815431}
associates an element of $K_0( \Groupoids \downarrow \Modules{R} )$ to
a finite domination; this defines the maps $\Wall$ and $\ReducedWall$
for a finitely dominated space, and Proposition~3.2 of \cite{MR815431}
yields

\begin{proposition}
  If a space $X$ is $R$-finitely dominated, and $\ReducedWall(X) = 0$,
  then $X$ is $R$-homotopy equivalent to a finite complex.
\end{proposition}
Or rather, we can only show that $\tilde{C}(X)$ is chain equivalent to
a complex of finitely generated free $R$-modules.  But in many cases, this
is enough: Leary (in Theorem 9.4 of \cite{MR1912650}), shows that if $G$
is a group of finite type (i.e., $BG$ has finitely many cells in each
dimension), then $G$ is $\FL(\Q)$ if and only if $G$ is $\FH(\Q)$.
This means that finite domination and the above algebra suffices to
get the geometry.

What we have done thus far for spaces is valid for
$\mathcal{C}$-spaces.  For instance, a $\mathcal{C}$-space $Y$ is said
to be $R$-finitely dominated if there is a finite $\mathcal{C}$-space
$X$, (meaning for each $c \in \mathcal{C}$, the space $X(c)$ is a
finite complex), and maps $Y \stackrel{i}{\longrightarrow} X
\stackrel{r}{\longrightarrow} Y$ with an $R$-homotopy equivalence $r
\circ i$.
\begin{proposition}
  If $Y$ and $Y'$ are $R$-finitely dominated contra- and covariant
  (respectively) $\mathcal{C}$-spaces and $\mathcal{C}$ is finite, then
$$
Y \times_{\mathcal{C}} Y'
$$
is an $R$-finitely dominated space.
\end{proposition}
\begin{proof}
  This is fairly straightforward: suppose $Y$ and $Y'$ are finitely
  dominated by $X$ and $X'$, respectively.  Then we have
$$
Y \times_{\mathcal{C}} Y'
\stackrel{i \times_{\mathcal{C}} i'}{\xrightarrow{\hspace*{1.5cm}}}
X \times_{\mathcal{C}} X'
\stackrel{r \times_{\mathcal{C}} r'}{\xrightarrow{\hspace*{1.5cm}}}
Y \times_{\mathcal{C}} Y'
$$
and it is enough to prove that 
$$
(r \times_{\mathcal{C}} r') \circ (i \times_{\mathcal{C}} i') : \tilde{C}(Y \times_{\mathcal{C}} Y') \to \tilde{C}(Y \times_{\mathcal{C}} Y')
$$
is an $R$-equivalence, and that $X \times_{\mathcal{C}} X'$ is a finite complex.

\end{proof}

The finiteness obstruction for a balanced product $Y
\times_{\mathcal{C}} Y'$ can be computed from the finiteness
obstructions of the terms $Y$ and $Y'$; this amounts to an equivariant
version of the Eilenberg--Zilber theorem, as in \cite{MR1697458}.
\begin{proposition}
  For $Y$ and $Y'$, finitely dominated contravariant and covariant
  $\mathcal{C}$-spaces, respectively,
$$
\Wall(Y \times_\mathcal{C} Y') = \Wall(Y) \otimes_\mathcal{C} \Wall(Y').
$$
\end{proposition}
This follows from the decomposition on page~229 of \cite{MR1027600},
relating balanced tensor product of chain complexes to the balanced
product of
spaces.

\section{Proof of Main Theorem}
\label{section:proof-of-main-theorem}

\subsection{A necessary condition}
\label{subsection:lefschetz}

Denote the Euler characteristics over a field $R$ by $\chi(X;R) =
\sum_i (-1)^i \dim H_i(X;R)$.  The Lefschetz Fixed Point Theorem
(\cite{MR1867354}, \cite{MR283793}) will obstruct some groups from
satisfying $\FH(R)$.
\begin{proposition}
\label{proposition:lefschetz-finiteness}

Suppose $R \supset \Q$ is a field, and that
$$
1 \to \pi \to \Gamma \to G \to 1,
$$
with $B\pi$ homotopy equivalent to a finite complex, and $G$ a finite
group.  If there exists a compact $X$ having $\pi_1 X = \Gamma$ and
$R$-acyclic $\tilde{X}$, then, for all nontrivial $g \in G$,
$$
\chi\left( \left(B\pi\right)^{\langle g \rangle} ; R\right) = 0.
$$
\end{proposition}
\begin{proof}
  The map $\tilde{X}/\pi \to B\pi$ is an $R$-homology equivalence, and
  is $G$-equivariant (though not necessarily an equivariant homotopy
  equialvence).  Consequently,
\begin{align*}
\chi\left(\left(B\pi\right)^{\langle g \rangle} ; R \right)
&= \trace \left(g_\star : H_\star(B\pi;R) \to H_\star(B\pi;R)\right) \hspace{1em}\mbox{(by Lefschetz)}\\
&= \trace \left(g_\star : H_\star(\tilde{X}/\pi;R) \to H_\star(\tilde{X}/\pi;R) \right) \hspace{1em}\mbox{(by $G$-equivariance)}\\
&= 0 \hspace{1em}\mbox{(by freeness of the $G$-action on $\tilde{X}/\pi$).}
\end{align*}
\end{proof}

Proposition~\ref{proposition:lefschetz-finiteness} is strong enough to
obstruct certain groups from satisfying $\FH(\Q)$.
\begin{example}
The group $\Z/n\Z$ acts on $\Z^n$ by permuting coordinates; $\Z/n\Z$
also acts on the kernel of the map $\Z^n \to \Z$ given by adding
coordinates.  Use the action on the kernel to define $\Gamma =
\Z^{n-1} \rtimes \Z/n\Z$.

The action of $\Z/n\Z$ on $B\Z^{n-1} = (S^1)^{n-1}$ fixes $n$ isolated
points.  Upon applying the classifying space functor $B$, the kernel
$\Z^{n-1}$ of the map $\Z^n \to \Z$ is
$$
B\Z^{n-1} = \{ (\alpha_1, \ldots, \alpha_n) \in (S^1)^{n} \mid \sum
\alpha_i = 0 \},
$$
and $\Z/n\Z$ acts on $B\Z^{n-1}$ by cycling coordinates.  So a fixed
point of the $\Z/n\Z$ is a point $(\alpha, \ldots, \alpha)$ with $n
\alpha = 0 \in S^1$.  There are $n$ such solutions, namely $\alpha =
2\pi k/n$ for $k = 0, \ldots, n-1$.  Consequently,
$$
\chi\left( \left( (S^1)^{n-1} \right)^{\Z/n\Z} \right) = n,
$$
and hence Proposition~\ref{proposition:lefschetz-finiteness} implies
that $\Gamma$ does not act freely on any $\Q$-acyclic complex.
\end{example}

\subsection{A sufficient condition}
\label{subsection:sufficient-condition}

The necessary condition given in
Proposition~\ref{proposition:lefschetz-finiteness} is not sufficient.
As we will see later, a sufficient condition requires examining the
Euler chacteristic of connected components of other subgroups, not
just the cyclic subgroups.  Here, we use the machinery from
Section~\ref{section:modules-over-categories} to prove the main
theorem.

Let $X$ be a $G$-space; we consider $X$ to be a contravariant
$\Or(G)$-space by
Proposition~\ref{proposition:recognize-spaces-with-actions}.

\begin{definition}
$B \Or(G)$ is the covariant $\Or(G)$-space given by
$$
G/H \mapsto BH = K(H,1),
$$
and sending the map $G/H \to G/H'$ to the map $BH \to BH'$ induced
from $H \subset H'$.
\end{definition}

\begin{proposition}
\label{prop:balanced-products}
For a $G$-space $X$,
$$
X \times_{\Or(G)} B\Or(G)
$$
is associated to the $G$-space $X \times_G EG = (X \times EG)/G$.
\end{proposition}
In light of Proposition~\ref{prop:balanced-products}, it may appear
that the machinery of balanced products did little but complicate the
usual Borel construction; as we will see in a moment, the machinery of
balanced products will facilitate the calculation of an instant
finiteness obstruction.

\begin{maintheorem}
\label{proposition:the-finiteness-calculation}
  The finiteness obstruction $\Wall(X \times_{\Or(G)} B\Or(G))$
  vanishes provided
$$
\chi(\mbox{connected component of $X^H$}) = 0
$$
for all nontrivial subgroups $H \subset G$.
\end{maintheorem}
\begin{proof}
  The covariant $\Or(G)$-space which $\Q$-finitely dominates $B\Or(G)$
  is a point over each orbit, and the contravariant $\Or(G)$-space
  which $\Q$-finitely dominates $X$ is simply $X^H$ over each orbit.
  Thus the balanced product is finitely dominated, and it remains to
  calculate the finiteness obstruction
  $$
  \Wall(X \times_{\Or(G)} B\Or(G)) =
  \Wall(X) \otimes_{\Or(G)} \Wall(B\Or(G)).
  $$
  Since each $H$ is finite, the rational chain complex of $BH$ can be
  taken to be the single module $[\Q]$ in degree 0, since $[\Q]$ is
  projective as a $\Q H$ module.  In short, $\Wall(B\Or(G))(G/H)$ is
  $[\Q]$.

  Since $X$ is already finite, $\Wall(X)$ is the equivariant Euler
  characteristic; the equivariant Euler characteristic as defined on
  page 99 of~\cite{MR1027600} is a chain complex over the component
  category $\Pi_0(G,X)$, which involves a contribution from each
  connected component of the fixed sets $X^H$.  Provided that the
  Euler characteristics of the components of fixed sets of $X$ vanish,
  then the balanced tensor product also vanishes, and so too the
  finiteness obstruction.
\end{proof}

\section{Examples}
\label{section:examples}

There are examples where the finiteness obstruction
vanishes; we give two sources of such examples: the reflection
group trick, and lattices with torsion.

\subsection{Reflection group trick}

We have already seen the reflection group trick (\cite{MR690848,
  MR2360474}) in section~\ref{section:reflection-group-trick}. To
further illustrate the trick, we construct a group $\Gamma$, with
$n$-torsion, which is the fundamental group of a rational homology
manifold (in fact, a manifold) with $\Q$-acyclic universal cover.
Since the fundamental group $\Gamma$ has $n$-torsion, there is no
closed, aspherical manifold with fundamental group $\Gamma$.
\begin{proposition}
  Let $G = \Z \times \Zmod{n}$.  Then there exists a closed manifold
  $M$ so that
\begin{itemize}
\item $\pi_1 M$ retracts onto $G$,
\item $\tilde{M}$ is $\Q$-acyclic.
\end{itemize}
\end{proposition}
\begin{proof}
  Let $\pi = \Z$, so that $B\Z = S^1$.  Consider $B\Z$ with the
  trivial $\Zmod{n}$ action, so that the fixed set $(B\Z)^{\Zmod{n}} =
  B\Z = S^1$ has vanishing Euler characteristic.  By the Main Theorem,
  there is a finite complex $Y$ having $\pi_1 Y = \Z \times \Zmod{n}$
  and having universal cover $\tilde{Y}$ rationally acyclic.  In other
  words, $G = \Z \times \Zmod{n} \in \FH(\Q)$.  For convenience, name
  the generators of $G$ by setting $G = \langle t, s \rangle s^n = 1,
  st = ts \rangle$.

  For this example we do not have to apply the Main Theorem to show $G
  \in \FH(\Q)$.  A construction of $Y$, side-stepping the Main
  Theorem, begins by showing that $G \in \FP(\Q)$, then that $G \in
  \FL(\Q)$, and finally that $G \in \FH(\Q)$, which yields the desired
  space.

  First, note that $\Q[\Z]$ is a projective $\Q[G]$-module, so 
  \begin{equation}
    \label{eqn:complex}
  \Q \longleftarrow \Q[\Z] \longleftarrow \Q[\Z] \longleftarrow 0
  \end{equation}
  is a finite length projective resolution of $\Q$ as a trivial
  $\Q[G]$-module; in other words, $G \in \FP(\Q)$.  To see that
  $\Q[\Z]$ is a projective $\Q[G]$-module, use the projection
  $$
  \proj_1(z) = \frac{1}{n} \left( 1 + s + \cdots + s^{n-1} \right) z
  $$
  which has image isomorphic to $\Q[\Z]$ on which $s$ acts trivially. Let
  $\Q[\Z]'$ be the image of the complementary projection,
  $$
  \proj_0(z) = \frac{1}{n} \left( (n-1) - s - + \cdots - s^{n-1} \right) z,
  $$
  so that $\Q[G] = \Q[\Z] \oplus \Q[\Z]'$.

  This projective resolution can be improved to a free resolution of
  $\Q$ as a trivial $\Q[G]$-module.  Tensor the
  resolution~(\ref{eqn:complex}) with the resolution
  $$
  \Q \longleftarrow \Q[\Zmod{n}] \longleftarrow \Q' \longleftarrow 0,
  $$
  where $\Q'$ is the $\Q[\Zmod{n}]$-module so that $\Q \oplus \Q'
  \cong \Q[\Zmod{n}]$.  This yields a complex
  $$
  \Q \longleftarrow \Q[G] \longleftarrow \Q[G] \oplus \Q[\Z] \longleftarrow
  \Q[\Z] \longleftarrow 0,
  $$
  By including a canceling pair of the complements $\Q[\Z]'$, we
  arrive at the desired free resolution
  $$
  \Q \stackrel{\epsilon}{\longleftarrow} \Q[G]
  \stackrel{d_1}{\longleftarrow}  \Q[G]^2
  \stackrel{d_2}{\longleftarrow}  \Q[G] 
  \longleftarrow 0.
  $$
  Here $\epsilon$ is the augmentation,
  \begin{align*}
    d_1(x,y) &= (1-t) x + \proj_1 y, \mbox{ and }\\
    d_2(z) &= (\proj_0 z, -(1-t) \proj_1 z + \proj_0 z).
  \end{align*}
  Thus, $G \in \FL(\Q)$.

  Theorem~9.4 of \cite{MR1912650} states that a group of
  finite type is $\FL(\Q)$ iff it is $\FH(\Q)$, so the free resolution
  in fact yields a space $Y$ having $\pi_1 Y = \Z \times \Zmod{n}$ and
  rationally acyclic universal cover $\tilde{Y}$.

  Since $Y$ is finite, it embeds in $\R^N$ for some $N$, and we can
  apply the reflection group trick to a regular neighborhood of $Y
  \subset \R^N$, producing a manifold (not merely a rational homology
  manifold) $M$ with universal cover $\tilde{M}$ rationally acyclic,
  and $\pi_1 M$ retracting onto $G$.
\end{proof}

\subsection{Preliminaries on lattices}

Historically, the first source of Poincar\'e duality groups were
fundamental groups of aspherical manifolds, and a basic source of
aspherical manifolds are lattices.
\begin{proposition}[\cite{MR145455}]
Let $G$ be a semisimple Lie group, $K$ a maximal compact, and $\Gamma$
a torsion-free uniform lattice (i.e., a discrete cocompact subgroup).
Then $\GmodKmodGamma$ is a compact, aspherical manifold with
fundamental group $\Gamma$.
\end{proposition}
\noindent
Consequently, uniform torsion-free lattices $\Gamma$ satisfy
$\PD(\Z)$.  Next, we examine what happens when $\Gamma$ is only virtually $\PD(\Z)$.
\begin{proposition}
\label{proposition:identify-pd}
Let $G$ be a finite group, $\pi$ a group satisfying $\PD(\Z)$, and $\Gamma$ an extension,
$$
1 \to \pi \to \Gamma \to G \to 1.
$$
Then $\Gamma$ satisfies $\PD(\Q)$.
\end{proposition}
\noindent
One can do better than $\Q$: if $R = \Z[1/G]$, meaning $\Z$
with divisors of $|G|$ inverted, then $\Gamma$ is $\PD(R)$.
\begin{proof}
  Extensions of Poincar\'e duality groups by Poincar\'e duality groups
  satisfy Poincar\'e duality \cite{MR311796}, and finite groups are
  $0$-dimensional $\Q$-Poincar\'e duality groups.
\end{proof}

Understanding groups satisfying which are virtually $\PD(\Z)$ permits
us to examine linear groups with torsion.
\begin{selberglemma}[\cite{MR130324}]
  Every finitely generated linear group contains a finite index normal
  torsion-free subgroup (in other words, every such group is virtually
  torsion-free).
\end{selberglemma}
\begin{example}
  Uniform lattices, even when they contain torsion, satisfy $\PD(\Q)$
  because, by Selberg's lemma, a uniform lattice is virtually
  torsion-free, and therefore, satisfies $\V\PD(\Z)$.
\end{example}

Whether $\chi(\GmodKmodGamma)$ vanishes is indepedent of $\Gamma$; it
depends only on the Lie group $G$.  This is true even if $\Gamma$ is
non-uniform (via measure equivalence \cite{MR1919400} and the equality
of the $L^2$ and usual Euler characteristic
\cite{MR420729,MR1926649}).  The fixed sets are themselves lattices in
smaller Lie groups, so it is easy to check that the Euler
characteristic vanishes on fixed sets.  As a result, lattices form a
particularly nice class with respect to the finiteness obstructions in
the Main Theorem.  In the next section, we produce some examples.

\subsection{Vanishing Euler characteristics of fixed sets}

\begin{proposition}
  \label{proposition:lattice-example}
  For odd $n$, there is a uniform torsion-free arithmetic lattice
  $\pi$ in $\SO(n,1)$ and a $\Zmod{n}$ action on the locally symmetric
  space
$$
X = \pi \backslash \SO(n,1) / \SO(n)
$$
with fixed set $X^{\Zmod{n}} = S^1$.
\end{proposition}
\begin{proof}
  We first recall the usual construction of arithmetic lattices; we
  follow Chapter~15C of \cite{arXiv:math.dg:0106063} and describe how
  to produce an arithmetic lattice in $\SO(n,1)$.  Begin by defining a
  bilinear form
$$
B(x,y) = \sum_{i=1}^n x_i\,y_i - \sqrt{2}\,x_0\,y_0
$$
so that $G = \SO(B) = \SO(n,1)$.  Note that $\Zmod{n}$ acts on
$\R^{n+1}$ preserving this form, and that the action is by integer
matrices.  Let $\mathcal{O} = \Z[\sqrt{2}]$, and let $G_{\mathcal{O}}$
denote the $\mathcal{O}$-points of $G$.

The diagonal map $\Delta : G \to G \times G^{\sigma}$, for $\sigma$
the Galois automorphism of $\Q(\sqrt{2})$ over $\Q$, sends
$G_{\mathcal{O}}$ to $\Delta(G_{\mathcal{O}})$, a lattice in $G \times
G^{\sigma}$.  But $G^{\sigma} = \SO(n+1)$ is compact, so after
quotienting by the Galois automorphism, $G_{\mathcal{O}}$ remains a
lattice in $G$.  And $G_{\mathcal{O}}$ is cocompact, by the Godement
Compactness Criterion (that arithmetic lattices are cocompact
precisely when they have no nontrivial unipotents \cite{MR141672}).
By Selberg's lemma, we may choose a finite-index subgroup of
$G_{\mathcal{O}}$; let $\pi$ denote the intersection of translates of
this finite-index subgroup under the $\Zmod{n}$ action.

The action of $\Zmod{n}$ on $\SO(n,1)$ descends to the quotient
$$X = \pi \backslash \SO(n,1) / \SO(n),$$
since it preserves the lattice $\pi$.  In the universal cover
$\SO(n,1) / \SO(n)$, the set fixed by $\Zmod{n}$ is a line; in the
quotient manifold $X$, the set fixed by $\Zmod{n}$ is no more than $1$
dimensional.  It is possible that the action might also have some
isolated fixed points---but there are no isolated fixed points,
because $\Zmod{n}$ cannot fix isolated points on an odd-dimensional
manifold (lest it freely act on the link, an even-dimensional sphere).
So the fixed set $X^{\Zmod{n}}$ is a $1$-manifold, i.e., a disjoint
union of circles.
\end{proof}
By Proposition~\ref{proposition:lattice-example}, there is an extension
$$
1 \to \pi \to \Gamma \to \Zmod{n} \to 1
$$
and since $\chi(B\pi^{\Zmod{n}}) = \chi(S^1) = 0$, the equivariant
finiteness theory implies that there exists a space $Y$ with $\pi_1 Y
= \Gamma$ and whose the universal cover $\tilde{Y}$ is a rationally
acyclic space.  In short, $\Gamma \in \FH(\Q)$.
\begin{question}
  \label{question:which-manifolds-have-group-actions}
  For which $n$ does $\Z/p\Z$ act with nontrivial fixed set on an
  hyperbolic $n$-manifold?
\end{question}
\noindent
This is possible in dimensions $2$ and $3$ by taking branched covers
(as in \cite{MR892185}).  Asking for a nontrivial fixed set is
important: Belolipetsky and Lubotzky \cite{MR2198218} have shown that
for $n \geq 2$, every finite group acts \textit{freely} on a compact
hyperbolic $n$-manifold.

In contrast, the construction in
Proposition~\ref{proposition:lattice-example} required dimension at
least $n$ to get $\Zmod{n}$ to act with nontrivial fixed set.  The
fact that there are only finitely many arithmetic triangle groups
\cite{MR429744} is perhaps relevant to answering Question
\ref{question:which-manifolds-have-group-actions}.

If we relax Question~\ref{question:which-manifolds-have-group-actions}
to a combinatorial curvature condition (i.e., locally $\CAT(-1)$; see
\cite{MR1744486}), we can easily prove the following.
\begin{proposition}
  \label{proposition:combinatorial-odd-dim-examples}
  For every $p$ and odd $n \geq 3$, there is a locally $\CAT(-1)$
  manifold $M$ admitting a $\Zmod{p}$ action having fixed set
  $M^{\Zmod{p}}$ a disjoint union of circles.
\end{proposition}
\begin{proof}
  In brief, first construct an action of $\Zmod{p}$ on a closed
  $n$-manifold $X^n$, having fixed set a disjoint union of circles, and
  finish by hyperbolizing.  Now we spell out a few details.

  Our construction of $X^n$ depends on our assumption that $n$ is odd;
  in this case, $\Zmod{p}$ acts freely on the odd-dimensional sphere
  $S^{n-2}$, and by taking the join with a circle on which $\Zmod{p}$
  acts trivially, we get an action of $\Zmod{p}$ on $S^n = S^{n-2}
  \ast S^1$ having fixed set $S^1$.

  Triangulate $X$ equivariantly; consequently, the fixed set $S^1$ is
  in the $1$-skeleton of $X$.

  Now apply strict hyperbolization \cite{MR1318879,MR1131435} to the
  triangulation of $X$.  The hyperbolized space inherits a $\Zmod{p}$
  action (since hyperbolization is functorial with respect to
  injective simplicial maps).  The $1$-skeleton of the hyperbolization
  of $X$ consists of two copies of $X^{(1)}$, so a fixed circle in $X$
  contributes two circles to the hyperbolization.

\end{proof}
We cannot do something similar for even dimensional manifolds (because
a $\Zmod{p}$ action with circle fixed set would give, by considering
the link of a fixed point, a $\Zmod{p}$ action on an odd dimensional
sphere with two fixed points, which is not possible).  But by
crossing the output of
Proposition~\ref{proposition:combinatorial-odd-dim-examples} with
$S^1$ on which $\Zmod{p}$ acts trivially, we produce an even
dimensional $\CAT(0)$ manifold with a $\Zmod{p}$ action fixing a
disjoint union of tori.  That is, we have shown
\begin{corollary}
  For every $p$ and every $n \geq 3$, there is a locally $\CAT(0)$
  manifold $M$ admitting a $\Zmod{p}$ action having non-empty fixed
  set with vanishing Euler characteristic.
\end{corollary}

\bibliographystyle{amsalpha}
\bibliography{references}

\end{document}